\documentclass[11pt, twoside]{article}

\usepackage[english]{babel}
\usepackage[T1]{fontenc} 
\usepackage[latin1]{inputenc} 

\usepackage{amsmath,amsthm,amscd,latexsym, mathtools, eucal,epsfig,enumitem,bbold, color} 
\usepackage{pstricks,pstricks-add,pst-math,pst-xkey,amssymb,stmaryrd,ulem}   
\usepackage[bookmarks={true},bookmarksopen={true}]{hyperref} 
 
\usepackage{bbm}

\newtheorem{theorem}{Theorem}
\newtheorem{proposition}[theorem]{Proposition}
\newtheorem{lemma}[theorem]{Lemma}

\theoremstyle{definition}
\newtheorem{definition}[theorem]{Definition}

\theoremstyle{theorem}
\newtheorem*{theorem*}{Theorem}
\newtheorem*{mainResult*}{Main Result} 
\newtheorem*{mainResultL*}{Main Result}
\newtheorem*{corollary*}{Corollary}

\newtheorem*{remark*}{Remark}
\theoremstyle{definition}
\newtheorem*{definition*}{Definition}
\begin{document}

\title{Exponential stable manifold for the synchronized state of the abstract mean field system}

\author{W. Oukil\\
\small\text{Faculty of Mathematics.}\\
\small\text{University of Science and Technology Houari Boumediene.}\\
\small\text{ BP 32 EL ALIA 16111 Bab Ezzouar, Algiers, Algeria.}}

\date{\today}

\maketitle

\begin{abstract}
  This paper investigates the exponential stability of abstract mean field systems in their synchronized state. We analyze stability by studying the linearized system and demonstrate the existence of an exponentially stable invariant manifold. Our focus is on the equilibrium stability under synchronization. We provide a comprehensive analysis of both linear and nonlinear cases of the system. Additionally, we prove the existence of stable limit cycles and establish a relation between the dynamics in linear and nonlinear frameworks.
\end{abstract}
\\




  \section{Introduction and Main results}
  This article presents an in-depth analysis of the stability and dynamics of coupled mean-field systems at the synchronized state. These systems, which   interact  in various scientific fields such as physics, biology, and social sciences   \cite{Min1, Min2, Min3} , exhibit complex behaviors due to their interactions. We explore the mathematical frameworks that govern these systems, focusing on the conditions under which stability can be achieved.   This research contributes to the growing body of literature on mean-field theory, offering both theoretical information and practical implications for understanding collective behavior in complex systems.
  
We study in this article the exponential stability on the mean field systems as the Winfree and Kuramoto mode in  \cite{Marco,Seung-Yeal, HaParkRyoo, kuramoto1,doi:10.1080/14689367.2018.1547683, WinfreeModel}   in its synchronized state.  In 1967 Winfree \cite{WinfreeModel}   proposed a mean field model  describing the synchronization of a population of organisms or {\it oscillators} that interact simultaneously \cite{articleariartnam}.   
 
Our main result consists of two parts: {\it The linear part}, where we will study the stability of a class of perturbed linear systems by decomposing the  fundamental matrix.  {\it The non-linear part},  where we study the exponential stability of th coupled mean-field systems. In other words, we shows the existence of a stable limit cycle. 

\subsection{  Linear part: Assumptions and results}
Consider the following notation: Let be $(n,m)\in\mathbb{N}^*\times\mathbb{N}^*$ and
\[A(t):=\{a_{i,j}(t)\in \mathbb{R},\ 1 \le i \le n,\ 1 \le j \le m,\ t \in I\subseteq \mathbb{R}\},\]
a $n\times m$ matrix-valued function, we denote the sup-norm of $A(t)$ as
\[||A||=\max\{\sup_{t\in I} |a_{ij}(t)|,1 \le i \le n,\ 1 \le j \le m\},\]
and we denote by $I_N$ the square identity matrix of order $N$.
\subsubsection{Stability Assumption $(H_{stab})$}
\label{abriviatHypSta}
Let   $\zeta(t)$ be a square matrix-valued function of  order $N$ that depends on time $t \in \mathbb{R}$ satisfying $||\zeta||<+\infty$. We consider the following perturbed linear system:
\begin{equation}\label{Chap1li}
\dot{Y}(t)=[b(t)I_{N}+\mathcal{A}(t)+\zeta(t)]Y(t),\quad t\geq t',
\end{equation}
where  $t'\in\mathbb{R}$ and  $\mathcal{A}(t)=\{a_{i,j}(t)\}_{1\le i,j\le N}$ is a continuous square matrix-valued function of order $N$ with rank 1, and $b :\mathbb{R}\to\mathbb{R}$ a continuous scalar function.
We consider the following assumption about $b(t)$ and $\mathcal{A}$ that we call {\it{ the stability assumption}}
\[
(H_{stab})\quad \left\{
    \begin{array}{ll}\text{ $a_{i,j}(t)=a_{j}(t)$ for all $1 \le i,j\le N$,} \\
\\
                             \text{$b :\mathbb{R}\to\mathbb{R}$, $a_j : \mathbb{R}\to \mathbb{R}$, $j=1,\ldots,N$,}\\
  \text{are a continuous  $1$-periodic functions},\\
\\
                     	\int_{0}^{1}b(s)+\sum_{j=1}^{N}a_j(s)ds=0,\quad\text{and}\quad-\alpha:=\int_{0}^{1}b(s)ds<0.
    \end{array}
\right.
\]
Define the {\it fundamental matrix} of a linear system  in the following sense
\begin{definition}
Let $t\mapsto {A}(t)$ be a continuous square matrix-valued function of order $N$. The {\it fundamental matrix} of the linear system
\[
\dot{y}=A(t)y,\quad t\in \mathbb{R},
\]
 is the matrix-valued function $R(s;t)$ with $s,t\in \mathbb{R}$ that satisfies
\begin{itemize}
\item  $\forall t_0 \in\mathbb{R}$, \ $R(t_0;t_0)=I_N$, 
\item  $\forall t \in\mathbb{R}$, \ $\frac{d}{dt}R(t;t_0) =A(t) R(t;t_0)$.
\end{itemize}
\end{definition}
To gain further insights into the behavior of solutions of the linear system \eqref{Chap1li}, we will introduce a class of matrices $\zeta$ referred to as "normalizing matrices," defined as follows 
\begin{definition}\label{normalizingproof}
We say that the matrix $\zeta$ is a {\it{normalizing}} matrix if the system \eqref{Chap1li} has a solution $V(t) = (v_1(t), \ldots, v_N(t))$ such that
\[
\inf_{t\in\mathbb{R}}||V(t)|| > 0, \quad \text{and} \quad \sup_{t\in\mathbb{R}}||V(t)|| < +\infty.
\]
We call $V(t)$ a {\it{normalizing}} solution of \eqref{Chap1li} associated with the matrix $\zeta$.
\end{definition}

\subsubsection{Linear Result  }\label{Chap1SectionHypothesisandmainresultslinear}
In the following two linear results, we will consider two cases: $\zeta$ being a normalizing matrix or arbitrary.
\begin{mainResultL*}[{{I}}$^l$]
Consider the system \eqref{Chap1li} with fundamental matrix $R(s;t)$. Suppose that $b$ and $\mathcal{A}$ satisfy the stability assumption ($H_{stab}$). Let $\beta\in(0,\alpha)$, then there exist $K > 0$ and $D_* > 0$ such that for any normalizing matrix $\zeta$ satisfying $||\zeta|| < D_*$ and for any $t\in \mathbb{R}$, there exists a linear form $\mathcal{L}_{t} : \mathbb{R}^N\to\mathbb{R}$ such that for any $Y\in \mathbb{R}^N$ and for any $s\geq t$, we have
\begin{description}
\item[$\bullet$] $\mathcal{L}_{s}(R(s;t)V(t))=1$ and $\mathcal{L}_{t}(Y)<K ||Y||$,
\item[$\bullet$] $\mathcal{L}_{t}(Y) =\mathcal{L}_{s}(R(s;t)Y)$,
\item[$\bullet$] $||R(s;t) [Y-\mathcal{L}_{t}(Y)V(t)]||<K ||Y||\exp(-\beta (s-t))$,
\end{description}
where $V(t)$ is a normalizing solution of \eqref{Chap1li} associated with the matrix $\zeta$. Moreover, the operator ${R}(s;t)$ admits the decomposition
\[
R(s;t)Y=\mathcal{L}_{t}(Y)V(s)+R(s;t) [Y-\mathcal{L}_{t}(Y)V(t)]  ,
\]
In other words, the submanifold $\mathcal{W}_{stab}:=\{Z\in\mathbb{R}^N, \mathcal{L}_t(Z)=0\}$ of dimension $N-1$ is exponentially stable.
\end{mainResultL*}

\begin{mainResultL*}[{{II}}$^l$][General Case]
Consider the system \eqref{Chap1li} with fundamental matrix $R(s;t)$. Suppose that $b$ and $\mathcal{A}$ satisfy the stability assumption ($H_{stab}$). Let $\beta\in(0,\alpha)$, then there exist $K > 0$ and $D_* > 0$ such that for any matrix $\zeta$ satisfying $||\zeta|| < D_*$ and for any $t\in\mathbb{R}$, there exists a linear form $\psi_{t} : \mathbb{R}^N\to\mathbb{R}$ such that for any $Y\in \mathbb{R}^N$ and for any $s\geq t$, we have
\[
||R(s;t) [Y-\psi_{t}(Y)\mathbbm{1}]||<K ||Y||\exp(-\beta (s-t)),
\]
where $\mathbbm{1}:=(1,\ldots,1)\in \mathbb{R}^N$. Moreover, the fundamental matrix ${R}(s;t)$ admits the decomposition
\[
R(s;t)Y=\psi_{t}(Y)R(s;t)\mathbbm{1}+R(s;t) [Y-\psi_{t}(Y)\mathbbm{1}].
\]
In other words, the submanifold $\mathcal{W}_{stab}:=\{Z\in\mathbb{R}^N, \psi_t(Z)=0\}$ of dimension $N-1$ is exponentially stable.
Furthermore, if there exists a solution that does not exponentially decay to zero, then
\begin{description}
\item[$\bullet$] $\psi_{t}(Y) =\psi_{s}(R(s;t)Y)$.
\end{description}
\end{mainResultL*}

\subsection{ Non-linear Part: Assumptions and results}\label{Chap2hypothesisnonlinearsynch}
The class of mean-field coupled systems we will study is given by the following two systems: the \textit{non-perturbed periodic system}
\begin{equation}\label{Chap1NotPerturbedSystem} \tag{NP}
\dot{x}_i=F(X,x_i),\quad i=1,\ldots,N,\quad t\geq t_0,
\end{equation}
and the \textit{perturbed system}
\begin{equation}\label{Chap1SystemGeneral} \tag{P}
\dot{x}_i=F(X,x_i)+H_i(X),\quad i=1,\ldots,N,\quad t \geq t_0,
\end{equation}
where $N\geq2$ and $X=(x_1,\ldots,x_N)$. The functions $F : \mathbb{R}^N \times \mathbb{R} \to \mathbb{R}$ and $H=(H_1,\ldots,H_N) : \mathbb{R}^N \to \mathbb{R}^N$ are $C^1$ functions.

\subsubsection{Notations and Definitions}
For $q,p \in \mathbb{N}^*$, let $g := (g_1, \ldots, g_p) : \mathbb{R}^q \to \mathbb{R}^p$ be a function. Consider the following semi-norm on the vector space of continuous functions from $\mathbb{R}^q$ to $\mathbb{R}^p$:
\[
||g||_{B} = \sup_{y \in B} \max_{1 \le i \le p} |g_i(y)|,
\]
where $B = \{y = (y_1, \ldots, y_q) \in \mathbb{R}^q : \max |y_i - y_j| \le 1\}$. This semi-norm is a norm on the vector space of continuous functions from $B$ to $\mathbb{R}^p$. We denote $d^i g$, $i = 1,2,\ldots$, as the $i^{\text{th}}$ derivative of $g$. We define
\[
||dg||_{B} = \max_{\substack{1 \le i \le p \\ 1 \le j \le q}} ||\partial_j g_i||_B, \quad ||d^2g||_{B} = \max_{\substack{1 \le i \le p \\ 1 \le j,k \le q}} ||\partial_k \partial_j g_i||_B.
\]
In the following, we will consider functions defined on $\mathbb{R}^q$ initially, with $q = N$, and then restrict them to the set $B$.

Let $g : \mathbb{R}^N \times \mathbb{R} \to \mathbb{R}$, $y = (y_1, \ldots, y_N) \in \mathbb{R}^N$, and $z \in \mathbb{R}$. We define
\begin{equation*}
\partial_i g(y,z) = \left\{
\begin{aligned}
&\frac{\partial}{\partial z} g(y,z), \quad i = N+1\\
&\frac{\partial}{\partial y_i} g(y,z), \quad i = 1,\ldots,N.
\end{aligned}
\right.
\end{equation*}
For $q,p \in \mathbb{N}^*$, a function $g : \mathbb{R}^q \to \mathbb{R}^p$ is called $\mathbbm{1}$-\textit{periodic} according to the following definition:
\begin{definition}\label{diagonalperiodicity}[$\mathbbm{1}$-periodicity]
Let $g : \mathbb{R}^q \to \mathbb{R}^p$ be a function, and denote $\mathbbm{1} := (1, \ldots, 1) \in \mathbb{R}^q$. The function $g$ is called $\mathbbm{1}$-\textit{periodic} if
\[
{g}(y + \mathbbm{1}) = {g}(y), \quad \forall y \in \mathbb{R}^q.
\]
\end{definition}

Notice that this definition does not imply that the function $g$ is periodic with respect to each variable individually.

Denote $\Phi^t$ as the flow of the system \eqref{Chap1SystemGeneral} (including \eqref{Chap1NotPerturbedSystem}). We now define the  positively $\Phi^t$-invariant set as follows:
\begin{definition}
Suppose that the flow $\Phi^{t}$ of the system \eqref{Chap1SystemGeneral} exists for all $t \geq t_{0}$. We say that an open set $C \subset \mathbb{R}^N$ is \textit{positively invariant} under the flow $\Phi^{t}$, or $C$ is positively $\Phi^t$-invariant, if $\Phi^{t}(C) \subset C$ for all $t \geq t_{0}$.
\end{definition}

\subsubsection{Synchronization Hypotheses $(H)$ and $(H_*)$}
We will consider the following hypotheses $(H)$ and $(H_*)$ about the field of the system \eqref{Chap1SystemGeneral} (including \eqref{Chap1NotPerturbedSystem}), which are sufficient for obtaining synchronization and stability.

\begin{align*}
(H) &\quad \left\{
    \begin{array}{ll}
    F \text{ is of class } C^2, \text{ and } \max\{||F||_{B}, ||dF||_{B}, ||d^2F||_{B}\} < +\infty,\\
    F \text{ is } \mathbbm{1}\text{-periodic and } \min_{s \in [0,1]} F(s\mathbbm{1},s) > 0,	   
    \end{array}
\right.\\
(H_*) &\quad \int_{0}^{1} \frac{\partial_{N+1} F(s\mathbbm{1},s)}{F(s\mathbbm{1},s)} \, ds < 0.
\end{align*}
 The hypothesis $(H_*)$ is known as the "synchronization hypothesis" when $(H)$ is verified. Hypothesis $(H)$ implies that the function $F(s\mathbbm{1},s)$ is uniformly Lipschitz. Specifically, in \cite{2017arXiv170307692O, doi:10.1080/14689367.2016.1227303}, the following Theorems {{I}} and {{II}} are stated. The first theorem shows the existence of a synchronization state, while the second shows a periodic locking state.

\begin{theorem*}[{{I}}]
Consider the system \eqref{Chap1SystemGeneral}. Suppose that $F$ satisfies hypotheses $(H)$ and $(H_*)$. Then, there exists $D_* > 0$ such that for all $D \in (0, D_*]$, there exists $r > 0$ and an open set $C_{r}$ of the form
\[
C_{r} := \Big\{ X = (x_i)_{i=1}^N \in \mathbb{R}^N \,:\, \exists \nu \in \mathbb{R}, \quad \max_{i} |x_i - \nu| < \Delta_{r}(\nu) \Big\},
\]
where $\Delta_{r} : \mathbb{R} \to (0, D]$ is a $C^1$ and $\mathbbm{1}$-periodic function, such that for any $H$ of class $C^1$ satisfying $||H||_{B} < r$, the following holds:

\begin{enumerate}
\item \textbf{Existence of Solution}: The flow $\Phi^t$ of the system \eqref{Chap1SystemGeneral} exists for all $X \in C_{r}$ and $t \geq t_{0}$.

\item \textbf{Synchronization}: The open set $C_{r}$ is positively invariant under $\Phi^t$. Furthermore, for all $X \in C_{r}$ we have $\min_{1 \le i \le N} \inf_{t \geq t_0} \frac{d}{dt} \Phi^t_i(X) > 0$ and
\begin{gather*}
 |\Phi^t_{i}(X) - \Phi^t_{j}(X)| < 2D,\ \forall 1 \le i, j \le N,\ \forall t \geq t_0.
\end{gather*} 
\end{enumerate}
\end{theorem*}
\begin{theorem*}[{{II}}]
Consider the system \eqref{Chap1SystemGeneral}. Suppose that $F$ satisfies hypotheses $(H)$ and $(H_*)$. Then, there exists $D_* > 0$ such that for all $D \in (0, D_*]$, there exists $r > 0$ such that for any function $H$ of class $C^1$ and $\mathbbm{1}$-periodic function satisfying $||H||_{B} < r$, there exists an open set $C_{r}$ (defined as in Theorem {{I}}) and an initial condition $X_* \in C_{r}$ such that
\[
\Phi^t_i(X_*) = \rho t + \Psi_{i,X_*}(t), \quad \forall i=1,\ldots,N, \ \forall t \geq t_0,
\]
where $\rho > 0$ and $\Psi = (\Psi_{1}, \ldots, \Psi_{N})$ with $\Psi_{i} : \mathbb{R} \to \mathbb{R}$ are $C^1$ and $\frac{1}{\rho}$-periodic functions.
\end{theorem*}
\subsubsection{Main Results}
The following result {{I}} shows the stability of the system \eqref{Chap1SystemGeneral}.
\begin{mainResult*}[{{I}}]
Consider the system \eqref{Chap1SystemGeneral}. Suppose that $F$ satisfies hypotheses $(H)$ and $(H_*)$. Then, there exists $D_* > 0$ such that for all $D \in (0, D_*]$, there exists $r > 0$, and there exists an open set $C_{r}$ (as defined in Theorem {{I}}), such that for any function $H$ of class $C^1$ satisfying $\max\{||H||_{B}, ||dH||_{B}\} < r$, we have
\begin{align*}
&\exists M > 0, \ \forall X \in C_{r}, \exists \delta > 0, \forall Y \in C_{r}, \ ||X - Y|| < \delta :\\
&||\Phi^{t}(X) - \Phi^{t}(Y)|| < M ||X - Y||, \ \forall t \geq t_0,
\end{align*}
Furthermore, for all $X \in C_{r}$, there exists a submanifold $\mathcal{W}_X \subset C_{r}$ of dimension $N-1$ at $X$ such that
\begin{align*}
&\exists \beta > 0, \ \exists K > 0, \ \forall X \in C_{r}, \ \forall Y \in \mathcal{W}_X: \\
&||\Phi^{t}(X) - \Phi^{t}(Y)|| < K \exp(-\beta (t - t_{0})) ||X - Y||, \quad \forall t \geq t_{0}.
\end{align*}
\end{mainResult*}
A consequence of Theorem {{II}} and Main Result {{I}} is the following result.
\begin{mainResult*}[{II}]
Consider the system \eqref{Chap1SystemGeneral}. Suppose that $F$ satisfies hypotheses $(H)$ and $(H_*)$. Then, there exists $D_* > 0$ such that for all $D \in (0, D_*]$, there exists $r > 0$, and there exists an open set $C_{r}$ (as defined in Theorem {{I}}) such that for any function $H$ of class $C^1$ and $\mathbbm{1}$-periodic satisfying $\max\{||H||_{B}, ||dH||_{B}\} < r$, there exists a submanifold $\mathcal{W}_{stab} \subset C_{r}$ of dimension $N-1$ such that
\begin{gather*}
\exists \beta > 0, \ \exists K > 0, \ \forall X \in \mathcal{W}_{stab}: \\
  ||\Phi^{t}(X) - \rho t - \Psi(t)|| < K \exp(-\beta (t - t_0)), \ \forall t \geq t_0,
\end{gather*}
where $\rho > 0$ and $\Psi = (\Psi_{1}, \ldots, \Psi_{N})$ with $\Psi_{i} : \mathbb{R} \to \mathbb{R}$ are $C^1$ functions and $\frac{1}{\rho}$-periodic.
\end{mainResult*}
\begin{remark*}
\begin{enumerate}
\item The last result {{II}} shows that the periodic orbit stated in Theorem {{II}} is a stable limit cycle.
\item In the general case, when the field is a trigonometric polynomial, the mathematical existence of the limits (frequencies)
\[
\lim_{t \to +\infty} \frac{x_i(t)}{t}, \quad i = 1, \ldots, n,
\]
is proved in \cite{doi:10.1080/14689367.2023.2170212}. This addresses the mathematical existence of equation (4) posed in \cite{10.1143/PTP.77.1005}.
\item Theorems {{I}} and Main Result {{I}} can be generalized for functions $H(t, X)$ that depend on time $t$.
\end{enumerate}
\end{remark*}

Throughout this paper, and to simplify the presentation, all proofs of the lemmas are provided in the appendix. 
\section{Proof of Linear Results {{I}}$^l$ and {{II}}$^l$: Stability of Perturbed Linear Systems}\label{Chap3GeneralLineaireStability}

To study the stability of coupled systems, as given by equation \eqref{Chap1SystemGeneral} (in particular \eqref{Chap1NotPerturbedSystem}), we will investigate the stability of a class of perturbed linear systems. More precisely, we will prove the two linear results {{I}}$^l$ and {{II}}$^l$ from the linear part.  These perturbed linear systems satisfy an assumption called the "stability assumption",  which is sufficient to decompose the phase space into a central manifold and another exponentially stable manifold.

In order to simplify the notation, we denote  $e(t,s):=\exp(\int_{s}^{t}b(x)dx)$. We denote by $\langle Y,Z \rangle $ the usual scalar product of $Y$ and $Z\in \mathbb{R}^N$, and by $Y^T$ the transpose of the vector $Y\in \mathbb{R}^N$. Recall that we denoted $\mathbbm{1}=(1,\ldots,1)^T\in \mathbb{R}^N$. Let's define

\begin{equation}\label{Chap3definitionA}
\left\{
    \begin{array}{ll}
        A_*(s)=(a_1(s),\ldots,a_N(s)), \\
       \zeta_i(s) =(\zeta_{i,1}(s),\ldots,\zeta_{i,N}(s)),\quad i = 1,\ldots,N.
    \end{array}
\right.
\end{equation}

\subsection{Proof Tools}\label{outildemonstrationstabilitylin}
To prove the linear results {{I}}$^{l}$ and {{II}}$^{l}$, we consider in this section only the stability assumption ($H_{stab}$). This allows us to deduce result {{II}}$^{l}$.
Consider the following non-homogeneous system defined for all $t\geq t'$:

\begin{equation}
	\label{Chap3equation:zi}\left\{
    \begin{array}{ll}
        \dot{Z}^*(t)&=b(t)Z^*(t) +\zeta(t)[z_{N+1}(t)\mathbbm{1}+Z^*(t)+e(t,t')Y],\\
         \dot{z}_{N+1}(t)&=[b(t)+\langle A_*(t),\mathbbm{1} \rangle ]z_{N+1}(t),\\
&+\langle A_*(t),Z^*(t)+e(t,t')Y \rangle, 
    \end{array}
\right.
\end{equation}
where $Z^*(t)=(z_1(t),\ldots,z_N(t))^T$, and $A_*$ is defined by equation \eqref{Chap3definitionA}. The goal of introducing the above system is that the part $E(t',t) = R(t;t') [Y-\psi_{t'}(Y)\mathbbm{1}]$ satisfies the following decomposition:
\[
E(t',t)Y=z_{N+1}(t)\mathbbm{1}+Z^*(t)+e(t,t')Y,
\]
where $Z(t):=(Z^*(t)^T,z_{N+1}(t))^T$ with $z_{N+1}(t)$ and $Z^*(t)=(z_{1}(t),\ldots,z_{N}(t))^T$ is a solution of the coupled non-homogeneous system \eqref{Chap3equation:zi}. The idea is to show that $||Z(t)||<K\exp(-\beta (t-t'))$ with $\alpha>\beta>0$. Here, it is worth noting that the initial condition of $Z(t)$ satisfies $z_{N+1}(t')\mathbbm{1}+Z^*(t')=-\psi_{t'}(Y)\mathbbm{1}$. The problem to solve is to find this suitable initial condition. To do this, we impose that $Z^*(t')=\psi_{t'}(Y)(\mathbbm{1}^T,0)^T$. First, we will see in this section under what assumptions we will have exponential decay of $Z(t)$ to zero. This will allow us to find a specific initial condition $\psi_{t'}(Y)(\mathbbm{1}^T,0)^T$, which in turn will lead to the linear form $\mathcal{L}_{t'}$ in Section \ref{Ingredientformelineaire}. We will need three lemmas, with the third Lemma \ref{Chap3lemmeprincipalestabilite} being the main one. The following lemma enables us to prove Lemma \ref{Chap3lemmepartieexpenontielle}, which in turn allows us to prove the main Lemma \ref{Chap3lemmeprincipalestabilite}.

\begin{lemma}\label{Chap3dispersionstabilite}
Let $b : \mathbb{R}\to\mathbb{R}$ be a periodic function such that
\[
-\alpha:=\int_{0}^{1}b(s)ds<0.
\]
Let $\alpha>\beta>0$, $L>0$, and $D>0$. Consider the following equation:
\[
\frac{d}{dt}\Delta(t) = [b(t)+\beta]\Delta(t)+ D L,
\]
then there exists $D_{0}>0$ such that for all $D<D_0$, the above equation has a solution $\Delta(t)$ that is $1$-periodic and strictly positive, with $\max_{t\in[0,1]}\Delta(t)<1$. The solution $\Delta(t)$ is given by
\[
\Delta(t)=D L \frac{\int_{t}^{t+1}\exp(\int_{s}^{t+1}b(x)+\beta dx)ds}{1-\exp(\beta-\alpha)}.
\]
\end{lemma}

\begin{lemma}\label{Chap3lemmepartieexpenontielle} Consider the system \eqref{Chap3equation:zi} with $b(t)$ and $\mathcal{A}$ satisfying the stability assumption ($H_{stab}$). Let $\beta\in(0,\alpha)$, then there exist $L>0$ and $D_*>0$ such that for any matrix $\zeta$ that is continuous and has a norm $||\zeta||<D_*$, for any $Y\in\mathbb{R}^N$, and for any solution $Z(t)=(Z^*(t)^T,z_{N+1}(t))^T$ of system \eqref{Chap3equation:zi} with an initial condition $Z(t')=Z\in\mathbb{R}^N$, we have:
\begin{align*}
\forall T>t'\ :\quad &z_{N+1}(T)=0\\
&\implies||Z(t)||<L\exp(-\beta (t-t'))[||Z||+||Y||],\ \forall t \in [t',T].
\end{align*}
\end{lemma}
\begin{proof}
Appendix. A
\end{proof}

The previous Lemma does not provide exponential decay of a solution over $[t',+\infty[$. To achieve this, the strategy in the following Lemma is to consider intervals $[0,T]$ from the previous Lemma, gradually increasing in size. We will approximate a solution $Z_Y(t)$ with an initial condition $Z_Y(t')$ using solutions that satisfy the previous Lemma. This allows us to demonstrate the existence of a solution $Z_Y(t)$ that exponentially decays to zero over $[t',+\infty[$. We will also localize the initial condition $Z_Y(t')$.

\begin{lemma}\label{Chap3lemmeprincipalestabilite} Consider the system \eqref{Chap3equation:zi} with $b(t)$ and $\mathcal{A}$ satisfying the stability assumption ($H_{stab}$). Let $\beta\in(0,\alpha)$, then there exist $K>0$ and $D_*>0$ such that for any $D<D_*$ and any $Y\in\mathbb{R}^N$, if there exists a sequence of solutions $Z_m(t)=(Z^*_m(t)^T,z_{N+1,m}(t))^T$ of \eqref{Chap3equation:zi} with initial condition $Z_{m}(t')=z_{t',m}W$ ($z_{t',m}\in\mathbb{R}$, $W=(\mathbbm{1}^T,0)^T$) and a sequence $(t_m)_m$ that tends to infinity such that $z_{N+1,m}(t_m)=0$, then there exists a solution $Z_{Y}(t)=(Z^*_Y(t)^T,z_{N+1,Y}(t))^T$ of system \eqref{Chap3equation:zi} with initial condition $Z_Y(t')=Z_Y\in \mathbb{R}^N$ such that
\[
||Z_Y(t)||<K\exp(-\beta (t-t'))||Y||, \quad \forall t \geq t'.
\]
Moreover, there exists a subsequence $(Z_{m_{k}}(t'))_{k}$ of $(Z_{m}(t'))_{m}$ such that $Z_Y(t')=\lim_{k\to+\infty}Z_{m_{k}}(t')$.
\end{lemma} 
\begin{proof}
Appendix. B
\end{proof} 
\subsection{Ingredients for the linear form $\mathcal{L}_{t'}$}\label{Ingredientformelineaire}
In this section, we will show the existence of a family of solutions to the system \eqref{Chap3equation:zi} that satisfies the assumptions of Lemma \ref{Chap3lemmeprincipalestabilite} in the previous section. To do this, we only need to determine the initial conditions $Z_{m}(t')=z_{t',m}W$. Note that it is sufficient to determine the sequence of real numbers $(z_{t',m})_m$. 

In this section, we consider $b(t)$ and $\mathcal{A}$ satisfying only the stability assumption $(H_{stab})$ without distinction in the matrix $\zeta$. This allows us, in particular, to deduce the linear result {{II}}$^l$.

To obtain more information about the solutions $Z_m(t)$ of Lemma \ref{Chap3lemmeprincipalestabilite}, we will express them in terms of the fundamental matrix. Let $S(t;t')=\{s_{i,j}(t;t')\}_{\substack{1\le i \le N+1 \\ 1\le j \le N}}$ be the fundamental matrix of the homogeneous linear system associated with the system \eqref{Chap3equation:zi} as follows:
\begin{equation}\label{Chap3systeme:homogenezi}
	\left\{
    \begin{array}{ll}
        \dot{X}^*(t)=b(t)X^*(t)+\zeta(t)[x_{N+1}(t)\mathbbm{1}+X^*(t)],\\
X^*(t)=(x_1(t),\ldots,x_N(t))^T,\\
        \dot{X}_{N+1}(t)=[b(t)+\langle A_*(t),\mathbbm{1} \rangle ]x_{N+1}(t)+\langle A_*,X^*(t) \rangle,
    \end{array}
\right.
\end{equation}
where we recall that $A_*(s)$ is given by Equation \eqref{Chap3definitionA}. In order to simplify the notation, denote $P(t,s)=\exp\Big{(}\int_{s}^{t}b(x)+\sum_{j=1}^{N}a_j(x)dx\Big{)}$. The solution $Z(t)=(Z^*(t)^T,z_{N+1}(t))^T$ of \eqref{Chap3equation:zi} with initial condition $Z(t')\in\mathbb{R}^N$ can be written in terms of the fundamental matrix as follows:
\begin{align}
\label{Chap3operateurZ}
Z^*(t)&=S^*(t;t')Z(t')+S_1(t;t')Y,\\
\label{Chap3equationintegralezN+1}
z_{N+1}(t)&=P(t,t')z_{N+1}(t')\\
\nonumber&+ P(t,t')\int_{t'}^{t}\langle A_*(s),S^*(s;t')Z(t')  \rangle P(t',s)ds\\
\nonumber&+P(t,t')\int_{t'}^{t}\langle A_*(s), S_1(s;t')Y+e(s,t')Y \rangle P(t',s)ds,
\end{align}
where $S^*(t;t')=\{s_{i,j}(t;t')\}_{\substack{1\le i \le N\\ 1\le j \le N}}$ and $S_1(t;t')$ is an operator that does not depend on $Z(t')$ and $Y$ such that $S(t',t')=0$.

As mentioned earlier, we aim to show the existence of solutions to the system \eqref{Chap3equation:zi} that satisfy the assumptions of Lemma \ref{Chap3lemmeprincipalestabilite}. In Lemma \ref{Chap3lemmeprincipalestabilite}, we have $Z_{m}(t')=z_{t',m}W$, so for a sequence $(T_m)_m$ tending to infinity such that $Z_{m,N+1}(T_m)=0$ and using the notation from the previous equation \eqref{Chap3equationintegralezN+1}, the sequence of real numbers $(z_{t',m})_m$ must be defined as
\[
z_{t',m}=-\frac{\int_{t'}^{T_{m}}\langle A_*(s),S_1(s;t')Y+e(s,t')Y \rangle P(t',s)ds}{\int_{t'}^{T_{m}}\langle A_*(s),S^*(s;t')W \rangle P(t',s)ds}.
\]
In Lemma \ref{Chap3lemme:formelineaires}, we will show that this sequence $(z_{t',m})_m$ is well-defined, i.e., the denominator of the quotient on the right-hand side of the equation is nonzero.

\begin{lemma}\label{Chap3lemme:formelineaires}
Let $S(t;t')=\{s_{i,j}(t;t')\}_{\substack{1\le i \le N+1 \\ 1\le j \le N}}$ be the fundamental matrix of the homogeneous linear system \eqref{Chap3systeme:homogenezi} associated with the system \eqref{Chap3equation:zi}. Let $S^*(t;t')=\{s_{i,j}(t;t')\}_{\substack{1\le i \le N\\ 1\le j \le N}}$. Suppose that $b(t)$ and $\mathcal{A}$ satisfy the stability assumption $(H_{stab})$. Define
\begin{align}\label{Chap3contradictionH}
H(t,t')&:= \int_{t'}^{t}\langle A_*(s),S^*(s;t')W \rangle P(t',s)ds,\\
\notag  W&:=(\mathbbm{1}^T,0)^T,\ t\geq t'.
\end{align}
Then, for any $\beta\in(0,\alpha)$, there exists $D_*>0$ such that for any continuous matrix $\zeta$ with norm $||\zeta||<D_*$, there exists $T_W>0$ such that for all $t\geq T_W$, we have $H(t,t')\neq0$.
\end{lemma}
\begin{proof} 
Appendix. C
\end{proof} 

Therefore, in the following proposition, we show that the system \eqref{Chap3equation:zi} admits a sequence of solutions that satisfy the assumptions of Lemma \ref{Chap3lemmeprincipalestabilite}.

\begin{proposition}\label{Chap3solutionpropositionlineare}
Consider the system \eqref{Chap3equation:zi} with $b(t)$ and $\mathcal{A}$ satisfying the stability hypothesis ($H_{stab}$). Let $\beta\in(0,\alpha)$. Then there exist $D_*>0$ and $K>0$, such that for any continuous matrix $\zeta$ with $||\zeta||<D_*$ and any $Y\in\mathbb{R}^N$, there exists a solution $Z_Y(t) = (Z^*_Y(t)^T, z_{N+1,Y}(t))^T$ of \eqref{Chap3equation:zi} such that
\[
||Z_Y(t)|| < K\exp(-\beta (t-t'))||Y||, \quad \forall t \geq t',
\]
and $Z_Y(t')=\psi_{t'}(Y)W$, where
\[
\psi_{t'}(Y) = \lim_{t_{m}\to+\infty}\frac{-1}{H(t_{m},t')} \int_{t'}^{t_{m}}\langle A_*(s), S_1(s;t')Y + e(s,t')Y \rangle P(t',s)ds,
\]
and where $(t_{m})_m$ is a sequence of real numbers that tends to infinity, and $H(t_m,t')$ is given by Lemma \ref{Chap3lemme:formelineaires}.
\end{proposition}

\begin{proof}
According to Lemma \ref{Chap3lemme:formelineaires}, for any sequence $(t_m)_m$ such that $t_m>T_W$, we have $H(t_m,t')\neq0$. From \eqref{Chap3equationintegralezN+1}, the solution $Z_{m}(t) = (Z^*_m(t)^T, z_{N+1,m}(t))^T$ of \eqref{Chap3equation:zi} with the initial condition $z_{m,t'} W$ such that
\[
z_{m,t'}=-\frac{1}{H(t_m,t')} \int_{t'}^{t_m}\langle A_*(s),S_1(s;t')Y+e(s,t')Y \rangle P(t',s)ds,
\]
satisfies:
\begin{align*}
z_{N+1,m}(t_{m})&=0\\
&=P(t',t) z_{m,t'}H(t_m,t')\\
&+ P(t',t) \int_{t'}^{t_m}\langle A_*(s),S_1(s;t')Y+e(s,t')Y \rangle P(t',s)ds.
\end{align*}
According to Lemma \ref{Chap3lemmeprincipalestabilite}, there exists a solution $Z_Y(t)$ such that $||Z_Y(t)|| < L \exp(-\beta (t-t'))||Y||$ for all $t\geq t'$, with $Z_Y(t')=\psi_{t'}(Y) W$, where
\begin{align*}
&\psi_{t'}(Y) = -\lim_{t_{m_{k}}\to\infty}\frac{\int_{t'}^{t_{m_{k}}}\langle A_*(s),S_1(s;t')Y+e(s,t')Y \rangle P(t',s)ds}{H(t_{m_{k}},t')}.
\end{align*}
\end{proof} 
\subsection{Decomposition of the fundamental matrix}
Finally, we will prove results {{I}}$^l$ and {{II}}$^l$ in the context of linear systems. To ensure consistency in the proof, we will first demonstrate the general case, which is the second linear result {{II}}$^l$ that does not require the matrix $\zeta(t)$ to be normalizing.

\begin{proof}[Proof of the second linear result {{II}}$^l$: General Case]~~\\
{\it Item 1.} Let's show that $||R(t;t')[Y+\psi_{t'}(Y)\mathbbm{1}]||$ decreases exponentially: For any $\beta>0$ and $||\zeta||=D<D_*$, where $D_*$ is defined by Lemma \ref{Chap3lemmeprincipalestabilite}, let $\psi_{t'}(Y)$ be given by the previous proposition \ref{Chap3solutionpropositionlineare}. Take $Y\in \mathbb{R}^N$, we add and subtract the same term in $\phi^t(Y)$ as follows:
\begin{align}\label{Chap3premieredecomposition}
\phi^t(Y)=-\psi_{t'}(Y) R(t;t') \mathbbm{1}+R(t;t')[Y+\psi_{t'}(Y)\mathbbm{1}].
\end{align}
Let $R(t;t')[Y+\psi_{t'}(Y)\mathbbm{1}]=z_{N+1}(t)\mathbbm{1}+Z^*(t)+e(t;t')Y$ with $z_{N+1}(t)$ being a solution with the initial condition $z_{N+1}(t')=0$ of the equation
\begin{align*}
\dot{z}_{N+1}(t)&=[b(t)+\langle A_*(t),\mathbbm{1} \rangle ]z_{N+1}(t)\\
&+\langle A_*(t),R(t;t')[Y+\psi_{t'}(Y)\mathbbm{1}]-z_{N+1}(t)\mathbbm{1} \rangle .
\end{align*}
Thus, $Z(t) = (Z^*(t),z_{N+1}(t))$ is a solution of the nonhomogeneous linear equation \eqref{Chap3equation:zi}. From equation \eqref{Chap3premieredecomposition}, we deduce that it has the initial condition $Z(t')=W \psi_{t'}(Y)$. Proposition \ref{Chap3solutionpropositionlineare} implies that
\[
||Z(t)|| < K\exp(-\beta(t-t'))||Y||, \quad \forall t \geq t'.
\]
Hence, $||R(t;t')[Y+\psi_{t'}(Y)\mathbbm{1}]|| < [K+\exp(c_b)]\exp(-\beta(t-t'))||Y||$ for all $t \geq t'$.\\
{\it Item 2.}  Let's show that when the system has a solution that does not decay to zero exponentially, then $\psi_{t}(Y) =\psi_{s}(R(s;t)Y)$:

Suppose that the system \eqref{Chap1li} has a solution that does not exponentially decay to zero. Let $V(t)$ be this solution. We have:
\begin{align*}
R(t;s)R(s;t')Y &= \mathcal{L}_{t'}(Y)V(t) - R(t;t') [Y-\mathcal{L}_{t'}(Y)V(t')]\\
&= \mathcal{L}_{s}(R(s;t')Y)V(t)\\
& - R(t;s) [R(s;t')Y-\mathcal{L}_{s}(R(s;t')Y)V(s)].
\end{align*}
By consequence, 
\begin{align*}
 [\mathcal{L}_{t'}(Y)& -\mathcal{L}_{s}(R(s;t')Y)]V(t) \\
&= R(t;s) [R(s;t')\mathcal{L}_{t'}(Y)V(t')+\mathcal{L}_{s}(R(s;t')Y)V(s)].
\end{align*}
Since the right-hand side of the last equation satisfies
\begin{align*}
&||R(t;s) [R(s;t')\mathcal{L}_{t'}(Y)V(t')+\mathcal{L}_{s}(R(s;t')Y)V(s)]||\\
& < K [\exp(-\beta(t-t'))+\exp(-\beta(t-s))] ||Y||.
\end{align*}
While the left-hand side is a linear form multiplied by the function $V(t)$, which does not decay exponentially to zero, we must have
\begin{align*}
\mathcal{L}_{t'}(Y) -\mathcal{L}_{s}(R(s;t')Y) = 0,\quad\forall t'\geq 0,\forall s\geq t'.
\end{align*} 
\end{proof}

Now, we will prove the first linear result {{I}}$^l$. We consider the particular case where $\zeta(t)$ is a normalizing matrix. By Definition \ref{normalizingproof}, the system \eqref{Chap1li} has a solution $V(t) = (v_1(t),\ldots,v_N(t))^T$ such that
\[
\inf_{t\in\mathbb{R}}||V(t)|| > 0,\quad\text{and}\quad \sup_{t\in\mathbb{R}}||V(t)|| < +\infty.
\]
We denote in the following
\[
\alpha_- = \inf_{t\in\mathbb{R}}||V(t)|| > 0,\quad\text{and}\quad \alpha_+ = \sup_{t\in\mathbb{R}}||V(t)|| < +\infty.
\]
We define in the proof below the linear form $\mathcal{L}_{t'}$ as
\[
\mathcal{L}_{t'}(Y) = \frac{\psi_{t'}(Y)}{\psi_{t'}(V(t'))},
\]
where $\psi_{t'}$ is defined by the previous proposition \ref{Chap3solutionpropositionlineare}. We note that by uniqueness, $\mathcal{L}_{t'}$ is defined as
\[
\mathcal{L}_{t'}(Y) = \lim_{t\to\infty}\frac{ \int_{t'}^{t}\langle A_*(s),S_1(s)Y+e(s,t')Y \rangle P(t',s)ds}{ \int_{t'}^{t}\langle A_*(s),S_1(s)V(t')+e(s,t')V(t') \rangle P(t',s)ds}.
\] 
\begin{proof}[Proof of linear result {{I}}$^l$]~\\
Let's prove $\mathcal{L}_{t}(R(t;t')V(t'))=1$:\\
By the definition of $\mathcal{L}_{t'}$, we have $\mathcal{L}_{t}(V(t))=1$ for all $t\in \mathbb{R}$. Therefore,
\begin{equation}\label{formelineairVinvariante}
\mathcal{L}_{t}(R(t;t')V(t'))=1=\mathcal{L}_{t'}(V(t')),\quad\forall t'\in\mathbb{R}\quad\forall t\geq t'.
\end{equation}
{\it Item 1.}  Construction of the linear form $\mathcal{L}_{t}$:\\
We add and subtract the same term in $V(t)$, so we have
\begin{equation}\label{Chap3relationV}
R(t;t') V(t')=V(t)=-\psi_{t'}(V(t')) R(t;t') \mathbbm{1}+R(t;t')[V(t')+\psi_{t'}(V(t'))\mathbbm{1}].
\end{equation}
According to proposition \ref{Chap3solutionpropositionlineare}, we have $||\psi_{t'}(V(t'))||<K ||V(t')||<K\alpha_{+}$ for all $t'\geq0$. Furthermore,
\begin{align*}
|\psi_{t'}(V(t'))| ||R(t;t') \mathbbm{1}||&=||V(t)-R(t;t')[V(t')+\psi_{t'}(V(t'))\mathbbm{1}]||\\
&>||V(t)||-||R(t;t')[V(t')+\psi_{t'}(V(t'))\mathbbm{1}]||\\
&>\alpha_{-}-K\exp(-\beta(t-t'))||V(t')||\\
&>\alpha_{-}-\alpha_{+}K\exp(-\beta(t-t')),\quad\forall t'\in\mathbb{R},.
\end{align*}
We integrate over a compact of length $\delta$ fixed such that $1<<\delta<+\infty$; let $t=t'+\delta$, we have from \eqref{Chap1li}: $||R(t;t')\mathbbm{1}||<\exp((c_b+c_a+D_*)\delta)$ where in order to simplify the notation we denoted $c_{b}=\max_{t\in[0,1]}|b(t)|$ and $c_{a}=\max_{t\in[0,1]}\sum_{j=1}^{N}|a_j(t)|$. 
\begin{equation}\label{Chap3minzV}
|\psi_{t'}(V(t'))| >\frac{\alpha_{-}-\alpha_{+}K\exp(-\beta\delta)}{\exp((2\alpha+c_b+c_a+D)\delta)}>0,\ \forall t'\in\mathbb{R},\ \forall t\geq t'.
\end{equation}
Let's define
\[
\mathcal{L}_{t'}(Y)=\frac{\psi_{t'}(Y)}{\psi_{t'}(V(t'))}.
\]
{\it Item 2.}  From \eqref{Chap3minzV} and according to proposition \ref{Chap3solutionpropositionlineare}, there exists $K_1>0$ such that $\mathcal{L}_{t'}(Y)< K_1 ||Y||$. \\
{\it Item 3.}  Exponential decay:\\
According to equation \eqref{formelineairVinvariante} for all $Y\in\mathbb{R}$
\begin{align*}
\phi^t(Y)&=\mathcal{L}_{t'}(Y)V(t)-R(t;t') [Y-\mathcal{L}_{t'}(Y)V(t')]\\
&=\mathcal{L}_{t'}(Y)V(t)\\
&+R(t;t') [Y-\psi_{t'}(Y)\mathbbm{1}-(\mathcal{L}_{t'}(Y)V(t')-\frac{\psi_{t'}(Y)}{\psi_{t'}(V(t'))}\psi_{t'}(V(t'))\mathbbm{1})]\\
&=\mathcal{L}_{t'}(Y)V(t)\\
&+R(t;t') [Y-\psi_{t'}(Y)\mathbbm{1}-(\mathcal{L}_{t'}(Y)V(t')-\psi_{t'}(\mathcal{L}_{t'}(Y)V(t'))\mathbbm{1})].
\end{align*}
Thanks to the previously demonstrated linear result {{II}}$^l$, there exists $K_*>0$ such that we have
\begin{align*}
&||R(t;t') [Y-\psi_{t'}(Y)\mathbbm{1}]||<K_*\exp(-\beta(t-t'))||Y||,\quad\forall t\geq t'\\
&||R(t;t') [\mathcal{L}_{t'}(Y)V(t')-\psi_{t'}(\mathcal{L}_{t'}(Y)V(t'))\mathbbm{1}]<K_*\exp(-\beta(t-t'))||Y||.
\end{align*}
{\it Item 4.}  Finally, let's show that $\mathcal{L}_{t'}(Y) = \mathcal{L}_{s}(R(s;t')Y)$:
We have
\begin{align*}
R(t;s)R(s;t')Y&=\mathcal{L}_{t'}(Y)V(t)-R(t;t') [Y-\mathcal{L}_{t'}(Y)V(t')]\\
&=\mathcal{L}_{s}(R(s;t')Y)V(t)\\
&-R(t;s) [R(s;t')Y-\mathcal{L}_{s}(R(s;t')Y)V(s)].
\end{align*}
Therefore,
\begin{align*}
[\mathcal{L}_{t'}(Y)& -\mathcal{L}_{s}(R(s;t')Y)]V(t)\\
&=R(t;s) [R(s;t')\mathcal{L}_{t'}(Y)V(t')+\mathcal{L}_{s}(R(s;t')Y)V(s)].
\end{align*}
Since $\min_{t\in\mathbb{R}}||V(t)||=\alpha_->0$ and as the right-hand side of the last equation verifies
\begin{align*}
||R(t;s) [R(s;t')&\mathcal{L}_{t'}(Y)V(t')+\mathcal{L}_{s}(R(s;t')Y)V(s)]||\\
&<K [\exp(-\beta(t-t'))+\exp(-\beta(t-s))] ||Y||.
\end{align*}
By consequence, $\mathcal{L}_{t'}(Y) -\mathcal{L}_{s}(R(s;t')Y)=0$ for all $t'\geq 0$ and $s\geq t'$. 
\end{proof}

\section{Proof of results {{I}} and {II}: Stability of mean-field systems}\label{Chap4GeneralNonLineaireStability}

In this section, we prove the two results {{I}} and {II}. We will study the stability of coupled systems given by equations \eqref{Chap1SystemGeneral} (in particular \eqref{Chap1NotPerturbedSystem}) and exhibiting a synchronized state. We linearize the system around a synchronized orbit and then apply the stability results obtained in Section \ref{Chap3GeneralLineaireStability}. Recall that the two systems \eqref{Chap1SystemGeneral} and \eqref{Chap1NotPerturbedSystem} are given by the following two equations:
\begin{equation}\label{Chap4NotPerturbedSystem1} \tag{NP}
\dot{x}_i=F(X,x_i),\quad 1\le i \le N,\quad t\geq t_0,
\end{equation}
\begin{equation}\label{Chap4SystemGeneral1} \tag{P}
\dot{x}_i=F(X,x_i)+H_i(X,x_i),\quad 1\le i \le N,\quad t \geq t_0,
\end{equation}
where $N\geq 2$ and $t_0 \in\mathbb{R}$ is the initial time. $F : \mathbb{R}^N\times\mathbb{R}\to\mathbb{R}$ and $H_i : \mathbb{R}^N\to\mathbb{R}$ are $C^1$ functions, and we denote $H=(H_1,\ldots,H_N)$. Recall that we denoted $\Phi^t$ the flow of the system \eqref{Chap1SystemGeneral} (including \eqref{Chap1NotPerturbedSystem}).

We have seen  \cite{2017arXiv170307692O,doi:10.1080/14689367.2016.1227303} that when the norm $||H||_B$ is sufficiently small and under the assumptions $(H)$ and $(H_*)$, the global phase dispersion of a solution of the system \eqref{Chap4SystemGeneral1} (and \eqref{Chap4NotPerturbedSystem1}) with an initial condition in the synchronization region remains uniformly bounded by some constant $D>0$.

\subsection{Linearization of System (P)}

In order to simplify the notation, define the function $F_i: \mathbb{R}^N\times \mathbb{R}\to \mathbb{R}$ as
\[
F_i(Y,z):=F(Y,z)+H_i(Y,z),\quad\forall Y\in\mathbb{R}^N,\ \forall z\in \mathbb{R},\quad\forall 1 \le 1 \le N,
\]
and denote the differential of the vector function $(F_1,\ldots,F_N)$ as $d\mathcal{F}(t)$.

Without loss of generality, for $Z \in \mathbb{R}^N$,  we'll write $Z:=\Phi^t(Z)$, so the elements of the matrix $dF(\Phi^t(Z)):=\{g_{i,j}\}_{i,j}$ are given by 
\begin{equation}
\label{Chap4element_jacobienne}
\left\{
\begin{aligned}
g_{ii}(t)&=\partial_{N+1}{F}_{i}(Z,z_i)+\partial_i {F}_{i}(Z,z_i),\ 1\le i \le N,\\ 
g_{ij}(t)&=\partial_{j}{F}_{i}(Z,z_i), 1\le i \neq j \le N.
\end{aligned}
\right.
\end{equation}

The strategy in the following is to apply the linear stability results obtained in Section \ref{Chap3GeneralLineaireStability}. To do this, we'll choose an appropriate linearization of the system \eqref{Chap4SystemGeneral1}. More specifically, for any $Z\in \mathbb{R}^N$, we consider the following linear system:
\begin{equation}\label{Chap4lin}
\frac{d}{dt}Y(t)=d\mathcal{F}(\Phi^t(Z))Y(t),\quad t\geq t_0,\quad Y(t)=(y_1(t),\ldots,y_N(t))^T.
\end{equation}

Also, recall that for $Z \in \mathbb{R}^N$, we denote $\mu_Z:=\mu_Z(t)$ as the solution of the system \eqref{Chap2position} associated with $\Phi^t(Z)$ and with the initial condition $\mu_Z(t_0)\in\mathbb{R}$ (see Definition \ref{SPN}).

In the following lemma, we'll see that the system \eqref{Chap4lin} can be written in the form of the linear systems \eqref{Chap1li}. We denote $I_N$ as the identity matrix of order $N$.

\begin{definition}\label{SPN}
For $X\in \mathbb{R}^N$, we define the {\it System} (SNP) {\it associated with} $\Phi^t(X)$ as the unperturbed system given by
\begin{align}\label{Chap2position}\tag{SNP}
\dot{\mu}_X = F(\Phi^t(X), \mu_X),\quad t\in I_X,
\end{align}
where $I_X = [t_{0}, T_X)$ is the maximal interval of the solution $X(t) := \Phi^t(X)$ of the system \eqref{Chap1SystemGeneral} with initial condition $\phi^{t_{0}}(X) = X$ fore some $t_0\in I_X$. We refer to $\mu_X(t)$ as the solution of the system \eqref{Chap2position} associated with $\Phi^t(X)$ with initial condition $\mu_X(t_{0})\in\mathbb{R}$.
\end{definition}
\begin{lemma}\label{lemmefinalstabilitynonlinear}
Consider the coupled nonlinear system \eqref{Chap4SystemGeneral1}. Suppose that $F$ satisfies hypotheses $(H)$ and $(H_*)$. Then, there exists $\epsilon_*>0$ such that for all $\epsilon\in(0,\epsilon_*]$, there exists ${r}>0$ such that for any function $H$ satisfying $\max\{||H||_B,||dH||_B\}<{r}$, there exists an open set $C_{r}$ that is positively $\Phi^t$-invariant, and for all $Z\in C_{r}$, there exists a diffeomorphism $t\to \mu=\mu_Z(t)$ such that the linearized system \eqref{Chap4lin} is equivalent to the following linear system:
\begin{gather}
\label{Chap4psifinal} \frac{d}{d\mu}Y^*(\mu)=[b(\mu)I_N+\mathcal{A}(\mu)+\zeta_Z(\mu)] Y^*(\mu),\\
\notag \mu\geq\mu_Z(t_0),\quad Y^*(\mu)=Y(\tau_Z(\mu)),
\end{gather}
where $\tau_Z :\mathbb{R}\to\mathbb{R}$ is the inverse function of $\mu(t)$, $\zeta_Z(\mu)$ is a normalizing matrix satisfying $||\zeta_Z||<\epsilon$, $\mathcal{A}(\mu)=\{a_{i,j}(\mu)=a_{j}(\mu)\}_{1\le i,j\le N}$ is a rank-1 matrix, and $b(\mu)$ is a function defined as follows:
\[
a_{j}(\mu)=\frac{\partial_j F(\mu\mathbbm{1},\mu)}{{F}(\mu\mathbbm{1},\mu)},\quad1 \le j\le N\quad\text{and}\quad b(\mu) = \frac{\partial_{N+1}{F}(\mu\mathbbm{1},\mu)}{{F}(\mu\mathbbm{1},\mu)},
\]
satisfying
\begin{equation*}
\int_{0}^{1}b(s)+\sum_{j=1}^{N}a_j(s)ds=0,\quad\text{and}\quad \int_{0}^{1}b(s)ds<0,
\end{equation*}
meaning that the stability hypothesis $(H_{stab})$ is satisfied.
\end{lemma}
\begin{proof}
Appendix. D
\end{proof}

\subsection{Proof of Main Result {(I)} and {(II)}: Stability}
The proof of Main Result {{I}} and  {{II}} relies on a perturbative method as follows:

\begin{proof}

Consider the coupled non-linear system \eqref{Chap4SystemGeneral1}, and suppose that $F$ satisfies hypotheses $(H)$ and $(H_*)$. According to Lemma \ref{lemmefinalstabilitynonlinear}, there exists $\epsilon_*>0$ such that for all $\epsilon\in(0,\epsilon_*]$, there exists ${r}>0$ such that for any function $H$ satisfying $\max\{||H||_B,||dH||_B\}<{r}$ for $Z\in C_{r}$, the system \eqref{Chap4lin} is equivalent to the system \eqref{Chap4psifinal} with $b(t)$ and $\mathcal{A}$ satisfying the stability hypothesis $(H_{stab})$, and $\zeta$ being a normalizing matrix satisfying $||\zeta_Z||<\epsilon$.

So, let $Z\in C_{r}$, and let $R_Z(\mu;\mu_Z)$ be the fundamental matrix of the system \eqref{Chap4psifinal}, where $\mu$ is the change of variable $t\to\mu_Z(t)$ with $\mu_Z(t)$ being a solution of the system \eqref{Chap2position} associated with $\Phi^t(Z)$ and with initial condition $\mu_Z:=\mu_Z(t_0)\in\mathbb{R}$ as defined in definition \ref{SPN}. According to the linear result {{I}}$^l$, we deduce that there exist $\beta>0$ and $K>0$, and there exists ${r}_*>0$ such that for a function $H$ satisfying $\max\{||H||_B,||dH||_B\}<{r}_*$ and for all $Z\in C_{{r}_*}$, there exists a linear form $\mathcal{L}_{\mu_{Z}} : \mathbb{R}^N\to\mathbb{R}$ such that for all $Y\in \mathbb{R}^N$ and for all $\mu\geq \mu_Z$, we have
\begin{description}
\item[$\bullet$] $\mathcal{L}_{\mu_{Z}}(R_Z(\mu;\mu_Z)V^*_Z(\mu_Z))=1$ and $\mathcal{L}_{\mu_{Z}}(Y)<K ||Y||$,
\item[$\bullet$] $\mathcal{L}_{\mu_{Z}}(Y) =\mathcal{L}_{\mu_{Z}}(R_Z(\mu;\mu_Z)Y)$,
\item[$\bullet$] $||R_Z(\mu;\mu_Z)  [Y-\mathcal{L}_{\mu_{Z}}(Y)V^*_Z(\mu_Z)]||<K ||Y||\exp(-\beta (\mu-\mu_Z))$.
\end{description}
Furthermore, the solution $Y^*(\mu)$ with initial condition $Y^*(\mu_Z)$ can be written as follows:
\begin{align*}
R_Z(\mu;\mu_Z)Y^*(\mu)&=\mathcal{L}_{\mu_{Z}}(Y^*(\mu_Z))V_Z^*(\mu)\\
&+R_Z(\mu;\mu_Z) [Y^*(\mu_Z)-\mathcal{L}_{\mu_{Z}}(Y^*(\mu_Z))V_Z^*(\mu_Z)].
\end{align*}
Let $S_{Z}(t;t_0)$ be the fundamental matrix of the system \eqref{Chap4lin}. Since the solution $Y(t)$ of the system \eqref{Chap4lin} with initial condition $Y(t')=Y$ satisfies $Y(t)=Y(\tau_Z(\mu))=Y^*(\mu)$, it follows that $Y^*(\mu_Z)=Y$ where $\mu\to\tau_Z(\mu)$ is the inverse of $t\to\mu_Z(t)$, then by the change of variable $\mu:=\mu_Z(t)\to t$, we obtain that $Y(t)$ can be written as follows:
\begin{align}
\nonumber Y(t)=S_{Z}(t;t_0)Y=Y^*(\mu)&=R_Z(\mu_Z(t);\mu_Z(t_{0}))Y\\
\label{resolvantestabilitevariablet}  &=\mathcal{L}_{\mu_{Z}}(Y)V_Z(t)\\
\notag &+R_Z(\mu_Z(t);\mu_Z(t_{0})) [Y-\mathcal{L}_{\mu_{Z}}(Y)V_Z(t)],
\end{align}
for all $ t\geq t_{0}$, where $V_Z(t)=V^*_Z(\mu_Z(t))=\frac{d}{dt}\Phi^t(Z)$. On the other hand, the system \eqref{Chap4SystemGeneral1} can also be written in the form:
\[
\frac{d}{dt}\Phi^{t}(Z)={F}(\Phi^{t}(Z))\implies \frac{d}{dt}{d \Phi^{t}(Z)}=d{F}(\Phi^{t}(Z)) {d \Phi^{t}(Z)}.
\]
Therefore, $d \Phi^{t}(Z)=S_Z(t;t_0)$.\\
\textit{Item 1. {Stability}}:

Let $X,Y\in C_{{r}_*}$ such that the segment $z(s)=(1-s)X+sY$ satisfies $z(s)\in C_{{r}_*}$ for all $s\in[0,1]$. We have:
\begin{align*}
\Phi^{t}(Y)-\Phi^{t}(X) &=\int_{0}^{1}{\frac{d \Phi^{t}(z(s))}{d s}}ds=\int_{0}^{1}{{d \Phi^{t}(z(s))}}ds (Y-X) \\
& = \int_{0}^{1}S_{z(s)}(t;t_0)ds (Y-X).
\end{align*}
According to the decomposition \eqref{resolvantestabilitevariablet} and knowing that $||V_{z(s)}||<||F||_B+ ||H||_B$, we obtain:
\begin{align*}
&||\Phi^{t}(Y)-\Phi^{t}(X) ||\le\Big|\Big|\int_{0}^{1}\mathcal{L}_{\mu_{z(s)}}(Y-X)V_{z(s)}(t)ds\Big|\Big|\\
&+\Big|\Big|\int_{0}^{1}R_{z(s)}(\mu_{z(s)}(t);\mu_{z(s)}(t_{0}))\Big{[}(Y-X)-\mathcal{L}_{\mu_{z(s)}}(Y-X)V_{z(s)}\Big{]}ds\Big|\Big|\\
&\le K [||F||_B+||H_B||]||Y-X||+K\exp\Big{(}-\beta(\mu_{z(s)}(t)-\mu_{z(s)}(t_0))\Big{)}||Y-X||\\
&\le K [||F||_B+||H_B||+1]||Y-X||,\quad\forall t\geq t_0.
\end{align*} 
\textit{Item 2. {Exponentially Stable Submanifold}}:

Let $X\in C_{{r}_*}$. In the following, we show that there exists an exponentially stable submanifold at $X$. The idea is to find trajectories of the form $z(s)$ connecting $z(0):=X$ and $z(1):=Y\in C_{r}$ such that for all $s\in [0,1]$, the quantity $\frac{d z(s)}{ds}$ lies in the kernel of the linear form $\mathcal{L}_{\mu_{z(s)}}$. This allows us to use equation \eqref{resolvantestabilitevariablet} and cancel the term that does not exponentially decay to zero as follows:
\begin{align}
\nonumber&||\Phi^{t}(Y)-\Phi^{t}(X)| =||\int_{0}^{1}{{d \Phi^{t}(z(s))}}ds \frac{d z(s)}{s}|| \\
\nonumber&\le  ||\int_{0}^{1}\mathcal{L}_{\mu_{z(s)}}(\frac{d z(s)}{s})V_{z(s)}(t)ds ||\\
\nonumber&+ || \int_{0}^{1}R_{z(s)}(\mu_{z(s)}(t);\mu_{z(s)})\Big{(}(\frac{d z(s)}{s})-\mathcal{L}_{\mu_{z(s)}}(\frac{d z(s)}{s})V_{z(s)}\Big{)}ds||\\
\nonumber&=|| \int_{0}^{1}R_{z(s)}(\mu_{z(s)}(t);\mu_{z(s)})\Big{(}(\frac{d z(s)}{s})-\mathcal{L}_{\mu_{z(s)}}(\frac{d z(s)}{s})V_{z(s)}\Big{)}ds||\\
\label{derniereinegaliteexponstable}&\le\exp(-\frac{\beta}{\alpha-L D} (t-t_0))\max_{s\in[0,1]}||\frac{d z(s)}{ds}||,
\end{align}
where $\alpha-L D>0$ since $\dot{\mu}_Z>\alpha-L D$ as given in the proof of the previous Lemma \ref{lemmefinalstabilitynonlinear}. To ensure that $\frac{d z(s)}{ds}$ lies in the kernel of the linear form $\mathcal{L}_{\mu_{z(s)}}$, it is sufficient that $z(s):=z(\xi,s)$ satisfies the following equation:
\begin{equation}\label{equationlineairezsxi}
\frac{d}{ds}z(\xi,s)=\xi-\mathcal{L}_{\mu_{z(\xi,s)}}(\xi)V_{z(\xi,s)}(t_{0}),
\end{equation}
with initial condition $Z(\xi,0)=X$ where $\xi\in\mathbb{R}^N$. For $||\xi||\approx0$, the solution $z(\xi,s)$ exists for all $s\in[0,1]$. In the rest of this proof, we consider vectors $\xi$ with norms close to zero such that $z(\xi,s)\in C_{{r}_{*}}$ for all $s\in[0,1]$. As $\mathcal{L}_{\mu_{Z}}(V_{z(\xi,s)}(t_{0}))=1$, we eventually obtain $\mathcal{L}_{\mu_{z(\xi,s)}}(\xi)(\frac{d}{ds}z(\xi,s))=0$, which implies equation \eqref{derniereinegaliteexponstable} and concludes the proof. We consider vectors $\xi \in \mathbb{R}^{N-1}$ in the hyperplane defined by $\mathcal{L}_{\mu_{X}}(\xi)=0$. We now show that the map $\xi \to \frac{d}{d\xi}z(\xi,1)\approx Id$ with $Id :\mathbb{R}^{N-1}\to\mathbb{R}^N$ is the identity matrix, i.e., a $C^1$-diffeomorphism, which defines the exponentially stable submanifold. We have:
\begin{align}
\notag\frac{d}{d\xi}\frac{d}{ds}z(\xi,s)&=Id_N -\frac{d}{d\xi}\Big{[}\mathcal{L}_{\mu_{z(\xi,s)}}(\xi)\Big{]}V_{z(\xi,s)}(t_{0})\\
\label{deriveunxi}&-\mathcal{L}_{\mu_{z(\xi,s)}}(\xi)\frac{d}{d\xi}\Big{[}V_{z(\xi,s)}(t_{0})\Big{]}.
\end{align}
The differentiability of the flow with respect to each variable implies that the linear form $\mathcal{L}_{\mu_{Z}}$ is differentiable with respect to $Z$. Since $\mathcal{L}_{\mu_{Z}}$ is linear with respect to $\xi$, by Taylor's theorem, there exists a matrix $\Pi$ of size $N-1\times N$ and $s'\in[0,s]$ such that:
\begin{align*}
\mathcal{L}_{\mu_{z(\xi,s)}}(\xi)&=\mathcal{L}_{\mu_{X}}(\xi)+\Pi\xi(z(\xi,s)-X)\\
&=\Pi\xi(z(\xi,s)-Z(\xi,0))=\Pi\xi\frac{d}{ds}z(\xi,s')s.
\end{align*}
Therefore, $\mathcal{L}_{\mu_{z(\xi,s)}}(\xi)$ is of second order with respect to $\xi$ and has the form: $\mathcal{L}_{z(\xi,s)}(\xi)\approx  \xi^2$. For $\xi\approx0$, there exists $M>0$ such that $||\frac{d}{d\xi}\mathcal{L}_{z(\xi,s)}(\xi)||<M||\xi||$. Considering that:
\[
\frac{d}{d\xi}\Big{[}V_{z(\xi,s)}(t_{0})\Big{]}=\frac{d}{d\xi}{F}(z(\xi,s))=d{F}(\Phi^t(z(\xi,s)))\frac{d}{d\xi}z(\xi,s),
\]
and as $||d{F}(\Phi^t(z(\xi,s)))||<L$, we can obtain $||\frac{d}{d\xi}z(\xi,s)||<M'$. Finally, there exists a constant $M''>0$ such that equation \eqref{deriveunxi} satisfies:
\[
||\frac{d}{d\xi}\frac{d}{ds}z(\xi,s)-Id_N||<M''\xi,
\]
for $\xi\approx0$. For such $\xi$, the function:
\[
z(\xi,1)=X+\int_{0}^{1}\frac{d}{d\xi}\frac{d}{ds}z(\xi,s)ds,
\]
is a $C^1$-diffeomorphism, which defines an exponentially stable submanifold $\mathcal{W}_{stab}$ of dimension $N-1$ such that the points in $\mathcal{W}_{stab}$ satisfy equation \eqref{derniereinegaliteexponstable}, i.e., for all $Y \in \mathcal{W}_{stab}$:
\[
||\Phi^{t}(Y)-\Phi^{t}(X)|| <K\exp(-\frac{\beta}{\alpha-L D} (t-t_0)),
\]
for all $t\geq t_0$, with $K>0$ being an upper bound of $K(\xi):=\max_{s\in[0,1]}||\frac{d }{ds}z(s,\xi)||$ for $\xi\approx0$, and $Y=z(\xi,1)$.
\end{proof}

  \section{Conclusion} 
In conclusion, the study of linear systems perturbed by a time-dependent matrix \(\zeta\) has established that \(\mathbb{R}^N\) can be decomposed into a direct sum of subspaces \(\mathbb{R} \oplus \mathcal{W}\), where \(\mathcal{W}\) is an exponentially stable manifold of dimension \(N-1\). For cases where the perturbing matrix \(\zeta\) is normalizing, we have derived information about the norm of the associated linear form. By relating these results to the stability of the nonlinear mean-field systems, we have demonstrated that the synchronization hypothesis and the stability hypothesis are closely related. Specifically, for the coupled system satisfying the synchronization hypothesis, linearization around synchronized orbits allows us to verify the stability hypothesis, showing the stability and exponential stability of the coupled system. In particular, when the perturbation of the coupled system is \(\mathbbm{1}\)-periodic, the exponential stability of the periodic orbit implies the existence of a stable limit cycle.

\appendix
\section*{Appendix. A}  
 \begin{proof}[Proof of the Lemma \ref{Chap3lemmepartieexpenontielle}] Let $Z(t) = (Z^*(t)^T, z_{N+1}(t))^T$, where $Z^*(t) = (z_{1}(t),\ldots,z_{N}(t))^T$ is a solution to the linear system \eqref{Chap3equation:zi} with the initial condition $Z(t') = Z\in\mathbb{R}^N$. Suppose that $z_{N+1}(T) = 0$. By integrating \eqref{Chap3equation:zi}, we obtain
\begin{equation}\label{Chap3valeurzN+1}
 z_{N+1}(t') = -\int_{t'}^{T}\langle A_*(s),Z^*(s)+e(s,t')Y \rangle P(t',s)ds.
\end{equation}
Let $\beta<\alpha$. Let $C={||Z||+||Y||}$, so for all $M>1$, we have $||Z||<MC$, and there exists $\epsilon>0$ such that $||Z^*(t)||<MC\exp(-\beta (t-t'))$
for every $t\in(t', t'+\epsilon)$. Define,
\[
T_{*}=\sup\Big\{t> t':\quad  ||Z^*(s)||<MC\exp(-\beta (s-t')),\quad \forall s\in(t',t)\Big\},
\]
Then, we have
\begin{equation}\label{Chap3approche:zim}
||Z^*(t)||<MC\exp(-\beta (t-t'))\quad\forall t\in [t',T_*[.
\end{equation}
The strategy is to find a particular constant $M$ such that $T_*\geq T$. By contradiction, assume that $T_{*}<T$. By integrating \eqref{Chap3equation:zi} and using \eqref{Chap3valeurzN+1}, we obtain for all $t\in [t',T_*[$
\begin{align}
\nonumber&|z_{N+1}(t)| =\Big{|}P(t,t')\Big{[}z_{N+1}(t')+\int_{t'}^{t}\langle A_*(s),Z^*(s)+e(s,t')Y \rangle P(t',s)ds\Big{]}\Big{|}\\
\nonumber&=P(t,t')\Big{|}-\int_{t}^{T}\langle A_*(s),Z^*(s)+e(s,t')Y \rangle P(t',s)ds\Big{|}\\
\nonumber&<\exp(c_b+c_a)c_aC\int_{t}^{T}M\exp(-(\beta s-t'))+\exp(-\alpha (s-t')+c_b)ds\\
\label{Chap3majorationzN+1m}&<\exp(2c_b+c_a)c_aC\frac{2M}{\beta}\exp(-\beta (t-t') ).
\end{align}
Let $\zeta$ be a continuous matrix with $||\zeta||=D<D_*$, then from equation \eqref{Chap3equation:zi}, we have the following two inequalities for $z_{i}(t)$ and $-z_{i}(t)$
 \begin{align}
 \nonumber\pm \frac{d}{dt}{z}_{i}(t)&<b(t)(\pm z_{i}(t))+DC[\exp(2c_b+c_a)c_a\frac{2M}{\beta}\\
\notag&+M+\exp(c_b))]\exp(-\beta (t-t'))\\
\notag&<b(t)(\pm z_{i}(t))+DMC[\exp(2c_b+c_a)\frac{2c_a}{\beta}\\
\label{Chap3normezt'm}&+1+\exp(c_b))]\exp(-\beta (t-t')).
\end{align}  
 Let $z_{i}(t)=\Delta_{i}(t)\exp(-\beta (t-t'))MC,\quad t\in [t',T_{*}[$.  We have $|\Delta_{i}(t)|\le1$ on $[t',T_{*}[$. Substituting into the last equation, we have
 \begin{align*}
\frac{d}{dt}(\pm\Delta_{i}(t))<[b(t)+\beta](\pm \Delta_{i}(t))+D[\exp(2c_b+c_a)\frac{2c_a}{\beta}+1+\exp(c_b))].
\end{align*} 
By definition of $T_{*}$ there exists $i\in \{1,\ldots,N\}$ such that $|\Delta_{i}(T_{*})|=1$. We will use Lemma \ref{Chap3dispersionstabilite} to obtain a contradiction.\\
Lemma \ref{Chap3dispersionstabilite} implies that for $||\zeta||=D<D_*:=D_0$, there exists a strictly positive $1$-periodic function of the equation
\[
\frac{d}{dt}\Delta(t)=[b(t)+\beta]\Delta(t)+D[\exp(2c_b+c_a)\frac{2c_a}{\beta}+1+\exp(c_b))],
\]
such that $\max_{t\in [0,1]}\Delta(t)<1$. Let $M>1$ be such that $\frac{1}{M}<\Delta(t')$; thus, $\Delta_{i}(t')\le\frac{||Z||}{MC}\le\frac{1}{M}<\Delta(t')$. There exists $\epsilon'>t'$ such that $|\Delta_{i}(t)|<\Delta(t)$ on $[t',\epsilon'[$; let $T_{0}=\sup\{t'\le s\le T_* : |\Delta_{i}(s)|<\Delta(s)\}$. If $t'<T_{0}<T_{*}$, then $|\Delta_{i}(T_{0})|=\Delta(T_{0})$. Without loss of generality, assume that $\Delta_{i}(T_{0})=\Delta(T_{0})$, then we obtain
\begin{align*}
\frac{d}{dt}|\Delta_{i}(T_{0})|&<[b(t)+\beta]|\Delta_{i}(T_{0})|+D[\exp(2c_b+c_a)\frac{2c_a}{\beta}+1+\exp(c_b))]\\
&=[b(t)+\beta]\Delta(T_{0})+D[\exp(2c_b+c_a)\frac{2c_a}{\beta}+1+\exp(c_b))]\\
&=\frac{d}{dt}\Delta(T_{0}).
\end{align*}
Contradiction. Thus, $T_{0}>T_{*}$, which implies $\Delta_{i}(T_{*})<1$, contradicting the definition of $T_{*}$. Therefore, for all $t\in[t',T]$, we have \[||Z^*(t)||<L \exp(-\beta (t-t'))[||Z(t')||+||Y||]\] for some constant $L =M\max(1,2\exp(2c_b+c_a)c_a)$.
\end{proof}

\section*{Appendix. B}  
\begin{proof}[Proof of the Lemma \ref{Chap3lemmeprincipalestabilite} ]
Let $Z_m(t)$ and $(t_m)_m$ satisfy the assumptions of this lemma. Lemma \ref{Chap3lemmepartieexpenontielle} implies that
\[
||Z_m(t)||<L\exp(-\beta (t-t'))[||Y||+||Z_m(t')||], \quad \forall t \in [t',t_m].
\]
The idea is to show that there exists $C>0$ such that $||Z_m(t')||<C||Y||$ and then extract a convergent subsequence of $Z_{m}(t)$. Let $\zeta$ be a continuous matrix with norm $||\zeta||=D<D_*$. By integrating \eqref{Chap3equation:zi} for all $t\in [t',t_m]$ and all $i=1,\ldots, N$, we have
\[
z_{i,m}(t)=e(t,t')z_{i,m}(t')+F_{i,m}(t),
\]
implies,
\begin{align*}
|F_{i,m}(t)|&=|e(t,t')\int_{t'}^{t}<\zeta_i(s),z_{N+1,m}(s)\mathbbm{1}+Z^*_m(s)+e(s,t')Y \rangle e(t',s)ds|\\
&<D(2L+\exp(c_b)) [||Y||+||Z_m(t')||]e(t,t')\int_{t'}^{t}\exp((\alpha-\beta) (s-t'))ds\\
&<D(2L+\exp(c_b))\exp(c_b) \frac{||Y||+||Z_m(t')||}{\alpha-\beta}\exp(-\beta (t-t')).
\end{align*}
By integrating \eqref{Chap3equation:zi}, we get $z_{N+1,m}(t_m) =0$ if and only if
\begin{align*}
 -\int_{t'}^{t_{m}}&\langle A_*(s),e(s,t')Y \rangle P(t',s)ds=\int_{t'}^{t_{m}}\sum_{i=1}^N [a_j(s)F_{i,m}(s)]P(t',s)ds\\
 & + z_{i,m}(t')\int_{t'}^{t_{m}}\langle A_*(s),\mathbbm{1} \rangle \exp(-\int_{t'}^{s}\langle A_*(x),\mathbbm{1} \rangle dx)ds.
\end{align*}
We deduce that
\begin{align*}
 &|z_{i,m}(t')|\Big{|}1-\exp(-\int_{t'}^{t_{m}}\langle A_*(x),\mathbbm{1} \rangle dx)\Big{|}\\
& -c_a D(2L+\exp(c_b))\exp(c_b) \frac{||Y||+|z_{t',m}|}{\alpha-\beta}\int_{t'}^{t_{m}}\exp(-\beta (s-t'))P(t',s)ds\\
&<|\int_{t'}^{t_{m}}\langle A_*(s),e(s,t')Y \rangle P(t',s)ds|.
\end{align*}
For $D\approx0$ and $m\to+\infty$, we will have
\[
|z_{i,m}(t')|<\frac{\exp(2c_b+c_a)}{1-c_aD(2L+c_b)\exp(2c_b+c_a) \frac{1}{\beta(\alpha-\beta)}}\Big{(} D\frac{c_a (2L+\exp(c_b))}{\beta(\alpha-\beta)}+\frac{c_a}{\alpha}\Big{)}||Y||,
\]
which implies that $(||Z_m(t)||)_m$ is uniformly bounded on each interval $[t',t_m]$. Furthermore,
\[
||Z_m(t)||<L\exp(-\beta (t-t'))[||Y||+||Z_m(t')||]<K\exp(-\beta (t-t'))||Y|| ,
\]
where
\[
K = L\Big{[}1+\frac{\exp(2c_b+c_a)\Big{(} D\frac{c_a (2L+\exp(c_b))}{\beta(\alpha-\beta)}+\frac{c_a}{\alpha}\Big{)}}{1-c_aD(2L+c_b)\exp(2c_b+c_a) \frac{1}{\beta(\alpha-\beta)}}\Big{]}.
\]
Therefore, we can extract a convergent subsequence that converges to a solution $Z_Y(t)$ of \eqref{Chap3equation:zi} and satisfies $
||Z_m(t)||<K\exp(-\beta (t-t'))||Y||$.
\end{proof}
\section*{Appendix. C}  
\begin{proof}[Proof of the Lemma \ref{Chap3lemme:formelineaires}]
From equation \eqref{Chap3equationintegralezN+1}, we deduce that $Z(t) = (S^*(t;t')W - z_{N+1}(t), z_{N+1}(t))$ is a solution of the linear homogeneous system \eqref{Chap3systeme:homogenezi} with the initial condition $W$, where $z_{N+1}(t)$ satisfies
\[
z_{N+1}(t) = P(t,t')H(t,t'), \quad t \geq t'.
\]
So, $Z(t) = (S^*(t;t')W - z_{N+1}(t), z_{N+1}(t))$ is, in particular, a solution of the nonhomogeneous linear equation \eqref{Chap3equation:zi} with $Y=0$ and satisfies the assumptions of Lemma \ref{Chap3lemmeprincipalestabilite}. Since $Y=0$, we have $||Z(t)|| \equiv 0$, which contradicts the initial condition $Z(t') = W$.
\end{proof}
\section*{Appendix. D}  
\begin{proof}[Proof of the Lemma \ref{lemmefinalstabilitynonlinear}] 
We consider the system \eqref{Chap4SystemGeneral1}. Suppose that $F$ satisfies hypotheses $(H)$ and $(H_*)$. Let $D\approx0$, then according to the Theorem  (I), there exists ${r}>0$ such that for a $C^1$  matrix $H$ satisfying $||H||_B<{r}$, there exists an open set $C_{r}$ that is positively $\Phi^t$-invariant, such that for all $t\geq t_0$ and all $Z\in C_{r}$, we have
\[
\max_{j}|\Phi_{j}^t(Z)-\mu_Z(t)|<D\quad\text{and}\quad\max_{i,j}|\Phi_i^t(Z)-\Phi_j^t(Z)|<2D,
\]
where $\mu_Z(t)$ is a solution of the system \eqref{Chap2position} associated with $\Phi^t(Z)$ and with the initial condition $\mu_Z(t_0)\in\mathbb{R}$. As long as $Z\in C_{r}$, note that   the flow $\Phi^t(Z)$ and the solution $\mu_Z(t)$ are defined for all $t\geq t_0$.

Consider the linearized system \eqref{Chap4lin} at the point $\Phi^t(Z)$. For any $1\le i,j\le N$, define $U_Z(t)=\{u_{i,j}^Z(t)\}_{1\le i, j \le N}$ as the matrix of order $N$ given by
\begin{equation*}
\left\{
\begin{aligned}
u_{ii}^Z(t)&= \Big{[}\partial_{N+1}H_i(Z,z_i)+[\partial_{N+1}{F}(Z,z_i)-\partial_{N+1}{F}(\mu_Z(t)\mathbbm{1},\mu_Z(t))]\Big{]}\\
&+\partial_i H_i(Z,z_i)+[\partial_i {F}(Z,z_i)-\partial_i {F}(\mu_Z(t)\mathbbm{1},\mu_Z(t))]\\
u_{i,j}^Z(t)&=\partial_j H_i(Z,z_i)+[\partial_j {F}(Z,z_i)-\partial_j {F}(\mu_Z(t)\mathbbm{1},\mu_Z(t))],
\end{aligned}
\right.
\end{equation*}
so that
\[
 \dot{Y}(t)=[d\mathcal{F}(\mu_Z(t)\mathbbm{1})+U_Z(t)] Y(t).
\]
Or equivalently,
\begin{align}\label{equationYnonfinal1}
\dot{y}_{i}(t)&=\partial_{N+1}{F}(\mu_Z(t)\mathbbm{1},\mu_Z(t))y_{i}(t)\\
\notag&+\sum_{j=1}^{N}[\partial_j F(\mu_Z(t)\mathbbm{1},\mu_Z(t))+u_{ij}^Z(t)]y_{j}(t).
\end{align}

Let $L>0$ and $\alpha>0$ be defined as follows:
\begin{equation}\label{Lalphadernierchapitre}
L:=\sum_{0 \le i \le 2}||d^iF||_B,\quad\text{and}\quad\alpha:=\min_{s\in [0,1]}F(s\mathbbm{1},s).
\end{equation}
Assume that $\max\{||H||_B,||dH||_{B}\}<{r}$, then for all $1\le i,j\le N$, we have
\begin{equation}\label{normeU}
||u_{i,i}^Z||<2[{r}+L D],\quad\text{and}\quad||u_{i,j}^Z||<{r}+L D.
\end{equation}
For $D\approx0$, we have $\dot{\mu}_Z>\alpha-L D>0$. In particular, $\mu_Z(t)$ is a diffeomorphism, and there exists $L>0$ such that
\begin{gather*}
\frac{F(\mu_Z\mathbbm{1},\mu_Z)}{\dot{\mu}_Z}=1+\theta(t),\\
|\theta(t)|=\big{|}\frac{F(\mu_Z\mathbbm{1},\mu_Z)-F(\Phi^t(Z),\mu_Z)}{\dot{\mu}_Z}\Big{|}<\frac{LD}{\alpha-L D}.
\end{gather*}
Let $\mu\to \mu=\mu_Z(t)$ be the change of variable and define $Y^*(\mu)=Y( \tau_Z(\mu))$, where $\tau_Z(\mu):=\mu_{Z}^{-1}(\mu)$ is the inverse function of $\mu(t)$. We can use the following equation in \eqref{equationYnonfinal1}:
\[
\frac{d}{dt}Y(t)=\frac{d}{ds}Y^*(s)\frac{d}{dt}\mu_Z(t)=\frac{d}{ds}Y^*(s)\frac{F(\mu_Z(t)\mathbbm{1},\mu_Z(t))}{1+\theta(t)}.
\]

Taking into account equations \eqref{normeU} and \eqref{equationYnonfinal1}, there exists a matrix $\zeta_Z(\mu)=\{\zeta_{ij}^Z(\mu)\}_{1\le i,j\le N}$ that depends on $\mu\in \mathbb{R}$, defined as
\[
\zeta_{ij}^Z(\mu)=d\mathcal{F}(\Phi^t(Z))\frac{\theta(t)}{F(\mu\mathbbm{1},\mu)}+u_{i,j}^Z(t)\frac{1}{F(\mu\mathbbm{1},\mu)},
\]
and satisfying
\[
||\zeta_Z||<(L+{r})\frac{||\theta||}{\alpha}+({r}+L D)\frac{1}{\alpha}<\epsilon:=(L+{r})\frac{LD}{\alpha-L D}+({r}+L D)\frac{1}{\alpha},
\] 
such that the system \eqref{Chap4lin} is equivalent to
\begin{align*}
\frac{d}{d\mu}{y}_{i}^*(\mu)=\frac{\partial_{N+1}{F}(\mu\mathbbm{1},\mu)}{{F}(\mu\mathbbm{1},\mu)}y_i^*(\mu)+\sum_{j=1}^{N}[\frac{\partial_j F(\mu\mathbbm{1},\mu)}{{F}(\mu\mathbbm{1},\mu)}+\zeta_{ij}^Z(\mu)]y_j^*(\mu),
\end{align*}
with $\mu\geq \mu_Z=\mu_Z(t_0)$. It can be observed that
\[
V_Z^*(\mu):=\Big{(}F_1(\Phi^{\mu_Z^{-1}(\mu)}(Z)),\ldots,F_N(\Phi^{\mu_Z^{-1}(\mu)}(Z))\Big{)}=\frac{d}{dt}\Phi^{\mu_Z^{-1}(\mu)}(Z),
\] 
is a solution of \eqref{Chap4psifinal} and satisfies
\begin{align*}
\min_{\mu\geq \mu_Z} V_Z^*(\mu) \geq \alpha- L D-{r}>0\quad \text{and}\quad ||V_Z^*||\le \max_{\mu\in[0,1]}F(\mu\mathbbm{1},\mu)+L D+{r}.
\end{align*}
Therefore, the matrix $\zeta_Z$ is a normalizing matrix. Furthermore, the two hypotheses $(H)$ and $(H_*)$ imply that 
\[
\int_{0}^{1}\frac{\partial_{N+1}{F}(\mu\mathbbm{1},\mu)}{{F}(\mu\mathbbm{1},\mu)}d\mu<0 \ \text{and}\ \int_{0}^{1}\frac{\partial_{N+1}{F}(\mu\mathbbm{1},\mu)}{{F}(\mu\mathbbm{1},\mu)}+\frac{\partial_j F(\mu\mathbbm{1},\mu)}{{F}(\mu\mathbbm{1},\mu)}d\mu=0.
\]
This implies that the stability hypothesis $(H_{stab})$ is satisfied.
\end{proof} 
\end{document}